\documentclass[11pt]{amsart}

\usepackage{amsfonts,amsmath,amssymb,amsthm}
\usepackage{color}
\usepackage[english]{babel}
\usepackage{cite}

\newtheorem{theorem}{Theorem}
\theoremstyle{plain}

\newtheorem{definition}{Definition}

\newtheorem{lemma}{Lemma}
\newtheorem{notation}{Notation}

\newtheorem{proposition}{Proposition}
\newtheorem{remark}{Remark}

\numberwithin{equation}{section}

\begin{document}
\title{Gradient  $\rho$-Einstein Solitons and  Applications}
\author{S\.INEM G\"uler}
\address[S. G\"uler]{Department of Industrial Engineering,
Istanbul Sabahattin Zaim University, Halkal\.i, Istanbul, Turkey.}
\email{sinem.guler@izu.edu.tr}
\author{B\"ulent \"Unal}
\address[B. \"{U}nal]{Department of Mathematics, Bilkent University,
Bilkent, 06800 Ankara, Turkey}
\email{bulentunal@mail.com}

\begin{abstract}
In this paper, we mainly study gradient  $\rho$-Einstein solitons on doubly
warped product manifolds. More explicitly, we obtain necessary and sufficient
conditions for a doubly warped product manifold to be a gradient $\rho$-Einstein soliton.
We also apply our main result to warped product spacetime models such as generalized
Robertson-Walker and standard static spacetimes as well as 3-dimensional Walker
manifolds. We finally establish that there is no 3-dimensional esentially
conformally symmetric gradient  $\rho$-Einstein soliton.
\end{abstract}

\maketitle

\section{Introduction}

The theory of Ricci solitons is a topic that has been studied extensively for the last twenty years.
This theory, in addition to being famous for  solving the Poincar\'e Conjecture proposed by Perelman, is a subject having a
wide range of uses in differential geometry and physical applications. One of the most important reasons why the Ricci flow theory is
important and worth studying is that it has self-similar solutions called as Ricci solitons and is often used to classify the
singularity models.

The motivation for the Ricci soliton and Ricci flow theory has paved the way for the study of different flow equations and their special solutions in the literature.
One of these new flow equations is  first introduced by Bourguignon \cite{bourguignon,Catino2016}, defined by the equation
$$\partial_g g = -2({\rm Ric-\rho g})$$ and named as Ricci-Bourguignon flow. Also, it is known that with appropriate rescaling over time,  when $\rho$ is not positive,  the Ricci-Bourguignon flow is an interpolation between the Ricci flow and the Yamabe flow, (for more, we refer \cite{Hamilton88,ye, Catino2017}).

%%%%%%%%%
As a   self-similar solution to the Ricci-Bourguignon flow, the gradient  $\rho$-Einstein soliton is defined as follows:
Let $(M^n,g)$ be an $n-$dimensional semi-Riemannian manifold and
$\rho$ is a real number. Then $(M^n,g)$ is said to be a   gradient $\rho$-Einstein soliton
if there exist a smooth function $\varphi$ on $M$ and a real number $\lambda$
such that
\begin{equation} \label{meqn:gradient-Yamabe}
{\rm Ric}+{\rm Hess}(\varphi)= (\rho \tau + \lambda)g,
\end{equation}
where $\tau$ is the scalar curvature of $M$. Here, the underlying
gradient $\rho$-Einstein soliton is denoted by $(M,g,\varphi,\rho,\lambda).$

Notice that if $\rho=0,$ then (\ref{meqn:gradient-Yamabe}) is reduced to the gradient Ricci soliton
equation. The gradient  $\rho$-Einstein solitons are said to be steady if $\lambda=0,$ shrinking if $\lambda>0,$ and
 expanding if $\lambda<0.$  Moreover, corresponding to special values of the parameter $\rho$, we refer to the $\rho$-Einstein solitons with different
names: The  $\rho$-Einstein soliton is called  an Einstein soliton if $\rho=1/2,$ a traceless soliton if $\rho=1/n,$ and
 a Schouten soliton if $\rho=1/2(n-1).$

Now, we will refer to some major studies in the field of gradient $\rho$-Einstein solitons. In \cite{Catino2015} the Catino, Mazzieri and Mongodi consider  shrinking gradient $\rho$-Einstein solitons and they classify non-compact gradient shrinkers with bounded non-negative sectional curvature. In \cite{Catino2016}, they also prove that every compact gradient Einstein, traceless Ricci, or Schouten soliton is trivial. Furthermore, they establish that every complete steady gradient Schouten soliton is trivial, and every complete 3-dimensional shrinking gradient Schouten soliton is isometric to a finite quotient of either
$\mathbb S^3$, $\mathbb R^3$ or $\mathbb R \times \mathbb S^2.$ By using algebraic curvature estimates and the Yamabe-Sobolev inequality, integral pinching rigidity results for compact gradient shrinking gradient $\rho$-Einstein solitons are proved in \cite{Huang2017}.
In \cite{Lemes2017}, Pina et al. proved that if $B \times _{f}F$ is a semi-Riemannian gradient Ricci soliton with potential
function $h$ then $f$ is constant or $h$ depends only on the base $B$.
In \cite{Mondal}, Mondal, Kumar and Shaikh obtain some sufficient conditions for a Riemannian manifold
admitting an almost $\eta$-Ricci soliton to be compact.
In \cite{Ho2020}, Ho considers Ricci-Bourguignon
flows of 3-dimensional locally homogeneous geometries on a closed 3-dimensional manifold endowed with evolution of Yamabe constants
and then he studies the complete Bach-flat shrinking gradient solitons as well. In \cite{Pina2020}, Pina and Menezes study the gradient $\rho$-Einstein solitons that are conformal
to a pseudo-Euclidean space and invariant under the action of the pseudo-orthogonal
group and then they provide all the gradient Schouten solitons for this type.
In \cite{Shaikh2021}, the authors established that a  gradient $\rho$-Einstein soliton with a
vector field of bounded norm and satisfying some other conditions is isometric to the Euclidean sphere.
In \cite{Blaga2021}, another version of gradient Ricci-Bourguignon soliton is
studied by means of the potential vector field. In \cite{Dwidedi}, the author showed that a compact gradient
Ricci-Bourguignon almost soliton is isometric to a Euclidean sphere if it has constant scalar curvature or
its associated vector field is conformal.

Very recently, in \cite{Shaikh2022}, the authors provide a lower bound of the diameter of a compact gradient  $\rho$-Einstein soliton under some special conditions. They conclude that some conditions of the
gradient  $\rho$-Einstein soliton with bounded Ricci curvature to be non-shrinking and non-expanding. Also, Blaga and Ta\c{s}tan study many types of gradient solitons including Riemann, Ricci, Yamabe and conformal on doubly warped product
manifolds, \cite{Blaga2022}. In \cite{G_U2022}, quasi-Yamabe gradient solitons on warped product manifolds are studied and the authors
prove the existence of the non-trivial gradient Yamabe soliton on generalized Robertson-Walker spacetimes, standard
static spacetimes, Walker manifolds and pp-wave spacetimes. In \cite{Turki2022}, Turki et al.  study gradient $\rho$-Einstein
solitons on warped product manifolds and investigate Einstein solitons on warped product manifolds admitting a conformal
vector fields.

Although the gradient $\rho$-Einstein soliton structure has been studied mainly on the warped product manifolds in the literature, it can be
said that this soliton structure has not yet been studied in depth for double warped product manifolds which is a much more general product metric and, even for Walker manifolds. Therefore, this paper is aimed to fill this gap in the literature and it is organized as follow. In Section 2, after we provide basic curvature related formulas for doubly warped products, we
derive our main results, that is we obtain necessary and sufficient conditions for a doubly warped product manifold to be a gradient $\rho$-Einstein soliton. In Section 3, we consider gradient $\rho$-Einstein solitons on warped product spacetime models and 3-dimensional Walker manifolds. We finally establish that there exists no 3-dimensional esentially conformally symmetric gradient  $\rho$-Einstein soliton.

\section{Doubly Warped Product gradient  $\rho$-Einstein Solitons}

Let $(M_1,g_1)$ and $(M_2,g_2)$ be two semi-Riemannian manifolds of dimensions $m_i$, $i=1,2$. Let  $\pi_1: M_1 \times M_2  \rightarrow M_1$
and  $\pi_2: M_1 \times M_2  \rightarrow M_2$ be the canonical projections.   Then doubly warped product manifold $(_{f_2}M_1 \times_{f_1}
M_2,g)$ of $(M_1,g_1)$ and $(M_2,g_2)$ is    the product manifold $M_1\times M_2$ equipped with the metric
\begin{equation} \label{e0}
g=(f_2 \circ  \pi_2)^2\pi_1^*g_1 +(f_1 \circ  \pi_1)^2\pi_2^*g_2
\end{equation}
where  $f_1: M_1 \rightarrow (0,\infty)$ and $f_2: M_2 \rightarrow (0,\infty)$ are called the warping functions of the doubly warped product
\cite{Bishop:1969,ehrlich,chen2017,unal}. Doubly warped product manifolds studied by different aspects. In \cite{gebarowski1,gebarowski2}, Gebarowski studies the divergence-free  and conformally recurrent doubly warped products. Also, in \cite{gupta18}, the non-existence of compact Einstein doubly warped product with non-positive scalar  curvature  is proved. We also refer to the recent article \cite{fathi} about the  doubly warped products of smooth
measure spaces to establish Bakry-\'Emery Ricci curvature (lower) bounds thereof in terms of those of the factors.

If one of the functions $f_1$ or $f_2$ is constant, then $(_{f_2}M_1 \times_{f_1} M_2,g)$ reduces to a warped product.
If both $f_1$ and  $f_2$ are constant, then we obtain a direct product manifold.

First, we will give the lemmas we need to prove our main results.

\begin{lemma}\label{lem:levi-civita DWP}
Let $X,Y \in \mathfrak{L}(M_1)$ and $U,V \in \mathfrak{L}(M_2)$. Then:
\begin{enumerate}
\item $\nabla_XY=\nabla^1_XY-g(X,Y) \nabla (\ln (f_2 \circ  \pi_2))$,
\item $\nabla_XV=\nabla_VX = V (\ln (f_2 \circ  \pi_2))X + X (\ln (f_1 \circ  \pi_1))V$,
\item $ \nabla_UV =\nabla^2_UV-g(U,V)\nabla  (\ln (f_1 \circ  \pi_1))$.
\end{enumerate}
\end{lemma}

\begin{remark}
From now on, we will denote $k = \ln f_1$ (respectively, $l = \ln f_2$) and use the same
symbol for the function $k$ (respectively, $l$) and its pullback $ k\circ \pi_1$ (respectively, $l \circ  \pi_2$).
\end{remark}

\begin{lemma}\label{lem:Hessian-1 DWP}
Let $\varphi$ be a smooth function on a doubly warped product of the form
$(_{f_2}M_1 \times_{f_1} M_2,g)$. For any $X,Y \in \mathfrak{L}(M_1)$ and $U,V
\in \mathfrak{L}(M_2)$, the Hessian of $\varphi$ satisfies:
\begin{enumerate}
\item  $({\rm Hess}\varphi) (X,Y)=({\rm Hess_1}\varphi)(X,Y)+g(\nabla l, \nabla \varphi)g(X,Y)$,
\item  $({\rm Hess}\varphi) (U,V)=({\rm Hess_2}\varphi)(U,V)
+g(\nabla k, \nabla \varphi)g(U,V)$,
\item  $({\rm Hess}\varphi) (X,U)=-X(k)U(\varphi)-X(\varphi)U(l)$,
\end{enumerate}
where $k=\ln f_1$ and  $l=\ln f_2$.
\end{lemma}

Thus, by using Lemma \ref{lem:Hessian-1 DWP}, we have:

\begin{lemma}\label{lem:Hessian-2 DWP}
For any $X,Y \in \mathfrak{L}(M_1)$ and $U,V \in \mathfrak{L}(M_2)$,
\begin{enumerate}
\item  $({\rm Hess}k) (X,Y)=({\rm Hess_1}k)(X,Y)$,
\item  $({\rm Hess}l) (X,Y)=g(\nabla l, \nabla l)g(X,Y)$,
\item  $({\rm Hess}k) (U,V)=g(\nabla k, \nabla k)g(U,V)$,
\item  $({\rm Hess}l) (U,V)=({\rm Hess_2}l)(U,V)$.
\end{enumerate}
where $k=\ln f_1$ and  $l=\ln f_2$.
\end{lemma}

\begin{lemma}\label{lem:Ricci DWP}
For any $X,Y \in \mathfrak{L}(M_1)$ and $U,V \in \mathfrak{L}(M_2)$, the components of the Ricci tensor of a
doubly warped product of the form $(_{f_2}M_1 \times_{f_1} M_2,g)$ are given by:
\begin{enumerate}
\item $ {\rm Ric}(X,Y)={\rm Ric}_1(X,Y)-\frac{m_2}{f_1}{\rm Hess}_{1}{f_1}(X,Y)-(\Delta l) g(X,Y), $ \\
\item ${\rm Ric}(X,V)=(m_1+m_2-2)X(k)V(l),$ \\
\item $ {\rm Ric}(U,V)={\rm Ric}_2(U,V)-\frac{m_1}{f_2}{\rm Hess}_{2}{f_2}(U,V)-(\Delta k) g(U,V), $ \\
\end{enumerate}
where  $k=\ln f_1$, $l=\ln f_2$ and $\Delta$ denotes the Laplacian on $(_{f_2}M_1 \times_{f_1} M_2,g)$ and
$m_i={\rm dim}(M_i)$, for $(i=1,2)$.
\end{lemma}

\begin{lemma}\label{lem:scalar curv DWP}
The scalar curvatures of a  doubly warped product manifold $(_{f_2}M_1 \times_{f_1} M_2,g)$ and its base $(M_1,g_1)$ and fiber $(M_2,g_2)$ are
related by:
\begin{align*}
    \tau=\frac{\tau_1}{f_2^2}+\frac{\tau_2}{f_1^2}-\frac{m_2}{f_1f_2^2}\Delta_1f_1-\frac{m_1}{f_2f_1^2}\Delta_2f_2-m_1\Delta l-m_2\Delta k,
\end{align*}
where $\tau_i$ denote the scalar curvatures of $(M_i,g_i)$ ($i=1,2)$,  $\Delta_i$ denotes the Laplacian on $M_i$, $(i=1,2)$.
\end{lemma}

\begin{theorem}
Let $(M=_{f_2}M_1 \times_{f_1} M_2,g)$ be a doubly warped product manifold admitting a gradient  $\rho$-Einstein soliton with the potential
function $\varphi$. Then:
\begin{itemize}
    \item[(1)] If $\varphi \in \mathcal{F}(M_1)$, then $(M,g)$ reduces to the warped product $(M_1\times_{f_1}M_2, \tilde{g_1}+f_1^2g_2)$ or
        $\varphi =(m_1+m_2-2)\ln f_1$.\\
    \item[(2)] If $\varphi \in \mathcal{F}(M_2)$, then $(M,g)$ reduces to the warped product $(_{f_2}M_1\times M_2, f_2^2{g_1}+\tilde{g_2})$
        or $\varphi =-(m_1+m_2-2)\ln f_2$. \\
    \item[(3)] If $(M=_{f_2}M_1 \times_{f_1} M_2,g)$ is  a non-trivial doubly warped product, then the potential function has to depend on
        both $M_1$ and $M_2$ and satisfies
    \begin{align}
        (m_1+m_2-2)X(k)U(l)-X(k)U(\varphi)-X(\varphi)U(l)=0,
    \end{align}
    for all $ X \in \mathfrak{L}(M_1) \ \textrm{and} \ U, \in \mathfrak{L}(M_2)$, where $\mathcal{F}(M_i)$ denote the set of all smooth
    functions on $M_i$ ($i=1,2$).
\end{itemize}
\end{theorem}

\begin{proof}
First note that if $(M,g)$ is a gradient $\rho$-Einstein soliton doubly warped product manifold, then
$${\rm Ric} + {\rm Hess}(\varphi) = (\rho \tau + \lambda) g \quad \text{where} \quad \rho, \lambda \in \mathbb R.$$
Substituting $X \in \mathfrak L(M_1),$ and $U \in \mathfrak L(M_2),$ yields:
\begin{equation} \label{thm1_eqn1}
(m_1+m_2-2)X(k)U(l) - X(k)U(\varphi) - X(\varphi)U(l)=0.
\end{equation}
\noindent (1) If the potential function $\varphi$ only depends on $M_1,$ then $U(l)=0.$
By equation (\ref{thm1_eqn1}),
$$U(l)[(m_1+m_2-2)X(k)-X(\varphi)] = 0.$$
Thus, \\
$\bullet$ If $U(l)=0$ for all $U \in \mathfrak L(M_2),$ then $l=\ln f_2$ is constant as
$l \in \mathcal F(M_2).$ So, $g$ can be simplified to a singly warped product metric. \\
$\bullet$ If $X(\varphi)=(m_1+m_2-2)X(k),$ for all $X \in \mathfrak L(M_1),$ then
$\varphi$ and $k = \ln f_1$ are proportional.

\noindent (2) If the potential function $\varphi$ depends only on $M_2,$ then $X(\varphi)=0,$
for any $X \in \mathfrak L(M_1)$ and the equation (\ref{thm1_eqn1}) yields,
$$X(k)[(m_1+m_2-2)U(l)-U(\varphi)]=0.$$
Thus, $X(k)=0$ or $U(\varphi)=(m_1+m_2-2)U(l).$ \\
$\bullet$ If $X(k)=0,$ for any $X \in \mathfrak L(M_1),$ then $l=\ln f_2$ is constant as
$l \in \mathcal F(M_2).$ So, $g$ can be simplified to a singly warped product metric. \\
\end{proof}

\begin{definition} \label{def:Gradient almost eta-Ricci soliton}
For the potential function $\varphi$, if the equation
\begin{align}
 {\rm Ric}+ {\rm Hess}(\varphi)=\gamma g + \mu \eta \otimes \eta
\end{align}
holds, where $\eta$ is the 1-form, $\gamma$ and $\mu$ are smooth functions,
the manifold is called gradient almost $\eta$-Ricci soliton
\cite{cho2019} and briefly denoted by $(M, g, \varphi, \eta, \mu, \gamma)$.
\end{definition}

\begin{theorem}
Let $(M=_{f_2}M_1 \times_{f_1} M_2,g)$ be a doubly warped product manifold.
Then $(M=_{f_2}M_1 \times_{f_1} M_2, g, \varphi, \rho, \lambda)$ is a gradient  $\rho$-Einstein soliton if and only if
\begin{itemize}
    \item[(1)] $(M_1, g_1, \varphi_1, \eta_1, \mu_1, \gamma_1)$ is a gradient almost $\eta$-Ricci soliton, where $
        \varphi_1=\varphi-m_2k$, \ $\eta_1=\nabla k$, \ $\mu_1=-m_2$ and $ \gamma_1=f_2^2[\rho \tau +\lambda +\Delta l-g(\nabla l, \nabla
        \varphi)]$.
    \item[(2)] $(M_2,g_2, \varphi_2, \eta_2, \mu_2, \gamma_2)$ is a gradient almost $\eta$-Ricci soliton, where $
        \varphi_2=\varphi-m_1l$, \ $\eta_2=\nabla l$, \ $\mu_2=-m_1$ and $ \gamma_2=f_1^2[\rho \tau +\lambda +\Delta k-g(\nabla k, \nabla
        \varphi)]$.
    \item[(3)] The potential function $\varphi$ satisfies
    \begin{align*}
        (m_1+m_2-2)X(k)U(l)-X(k)U(\varphi)-X(\varphi)U(l)=0,
    \end{align*}
     for all $ X \in \mathfrak{L}(M_1) \ \textrm{and} \ U, \in \mathfrak{L}(M_2)$, where $\mathcal{F}(M_i)$ denote the set of all smooth
     functions on $M_i$ ($i=1,2$).
\end{itemize}
\end{theorem}

\begin{proof}
Assume that $(M,g)$ is a gradient $\rho$-Einstein soliton doubly warped product manifold with
potential function $\varphi.$ Then by using the fundamental formulas stated in previous
lemmas, equation (\ref{meqn:gradient-Yamabe}) takes the following form:
\begin{equation}
\begin{split}{\rm Ric}^1(X,Y)-\frac{m_2}{f_1}{\rm h}_1^f(X,Y) \\
& - f_2^2 \Delta(l) g_1(X,Y) \\
& + {\rm h}_1^{\varphi}(X,Y) +  f_2^2 g(\nabla l, \nabla \varphi) g_1(X,Y) \\
& = f_2^2 [\rho \tau + \lambda]g_1(X,Y).
\end{split}
\end{equation}
Note that for any smooth positive function $f,$ we have the following well-known identity:
$${\rm h}^{\ln f} = \frac{1}{f} {\rm h}^f - \frac{1}{f^2} {\rm d}f \otimes {\rm d}f.$$
Hence,
$${\rm Ric}^1-m_2({\rm h}^k + {\rm d}k \otimes {\rm d}k) + {\rm h}_1^{\varphi} = \Lambda_1 g_1,$$
where $\Lambda_1 = f_2^2[(\rho \tau + \lambda) + \Delta l - g(\nabla l, \nabla \varphi)]$.
Then $${\rm Ric}^1 + {\rm h}_1^{\varphi - m_2 k} - m_2 {\rm d}k \otimes {\rm d}k = \Lambda_1 g_1,$$
i.e, $(M_1,g_1)$ is a gradient $\eta_1$-Ricci soliton where
$\psi_1=\varphi-m_2 k$ is a potential function, $\mu_1=-m_2$ and $\eta_1=\nabla k.$
Similarly, the fundamental formulas stated in previous
lemmas, equation (\ref{meqn:gradient-Yamabe}) takes the following form:
\begin{equation}
\begin{split}{\rm Ric}^2(V,U)-\frac{m_1}{f_2}{\rm h}_2^f(V,U) \\
& - f_1^2 \Delta(k) g_2(V,U) \\
& + {\rm h}_2^{\varphi}(V,U) +  f_1^2 g(\nabla k, \nabla \varphi) g_2(V,U) \\
& = f_1^2 [\rho \tau + \lambda]g_2(V,U).
\end{split}
\end{equation}
Thus,
$${\rm Ric}^2-m_2({\rm h}^l + {\rm d}l \otimes {\rm d}l) + {\rm h}_2^{\varphi} = \Lambda_2 g_2,$$
where $\Lambda_2 = f_1^2[(\rho \tau + \lambda) + \Delta k - g(\nabla k, \nabla \varphi)]$.
Then $${\rm Ric}^2 + {\rm h}_2^{\varphi - m_1 l} - m_1 {\rm d}l \otimes {\rm d}l = \Lambda_2 g_2,$$
i.e, $(M_2,g_2)$ is also a gradient $\eta_2$-Ricci soliton where
$\psi_2=\varphi-m_1 l$ is a potential function, $\mu_2=-m_1$ and $\eta_2=\nabla l.$
\end{proof}

%\section{Warped Product gradient  $\rho$-Einstein Solitons}
If one of the functions $f_1$ and  $f_2$ is constant, then $(_{f_2}M_1 \times_{f_1} M_2,g)$ reduces to a warped product.
Thus, we may consider the warped product  $M=B \times _{b}F$ endowed with the metric

%Assume that $(B,g_B)$ and $(F,g_F)$ are two pseudo-Riemannian manifolds of dimensions $r$ and $s,$ respectively. Let $\pi:B \times F \rightarrow B$ and $\sigma: B \times F \rightarrow F$ be the natural projection maps of the Cartesian product $B \times F$ onto $B$ and $F,$ respectively. Also, let $b:B \rightarrow \left( 0,\infty \right) $ be a positive real-valued smooth function. The warped product manifold $M=B\times _{b} F$ is the product manifold $B \times F$ equipped with the metric tensor defined by
\begin{equation}
\label{warped metric}
g=\pi^{\ast }\left( g_{B}\right) \oplus \left( b \circ \pi \right)
^{2}\sigma^{\ast }\left( g_{F}\right),
\end{equation}%
where $^{\ast }$ denotes the pull-back operator on tensors\cite%
{Bishop:1969,Oneill:1983,Shenawy:2015}. The function $b$ is called the
warping function of the warped product manifold $B \times _{b} F$, and
the manifolds $B$ and $F$ are called base and fiber, respectively. In
particular, if $b=1$, then $B \times _{1}F = B \times F$ is the
usual Cartesian product manifold. For the sake of simplicity,
throughout this paper, all relations will be written, without
involving the projection maps from $B \times F $ to each component
$B$ and $F$ as in $g=g_{B}\oplus b^{2}g_{F}$.

\begin{proposition} Let $M=B \times _{b}F$ be a singly warped
product manifold with the metric tensor $g=g_B \oplus b^2 g_F.$
$$
\tau  =  \tau_B + \frac{\tau_F}{b^2} -2s \frac{\Delta_B(b)}{b}
- s (s-1) \frac{g_B(\nabla^B(b), \nabla^B(b))}{b^2},
$$
where $r={\rm dim}(B)$ and $s={\rm dim}(F).$
\end{proposition}

\begin{notation} Suppose that  $M=B \times _{b}F$ is a warped
product manifold with the metric tensor $g=g_B \oplus b^2 g_F.$ Then
$$ b^\sharp = b \Delta^B(b) +(s-1) g_B(\nabla^B(b), \nabla^B(b)). $$
\end{notation}

\begin{theorem} \label{main-1} Let $M=B \times _{b}F$ be a
warped product manifold equipped with the metric
$g=g_{B}\oplus b^{2}g_{F}.$  Then $(M,g,\varphi,\rho,\lambda)$
is a gradient $\rho$-Einstein soliton if and only if the followings hold:
\begin{enumerate}
\item the potential function $\varphi$ depends only on the base manifold $B,$
\item the scalar curvature $\tau_F$ of the fiber manifold $(F,g_F)$ is constant,
\item $\displaystyle{{\rm Ric}_B + {\rm Hess}^{\varphi}_B =
(\rho \tau + \lambda)g_B + \frac{s}{b}{\rm Hess}^b_B}$
\item $\displaystyle{{\rm Ric}_F = [b^\sharp - bg_B(\nabla^B(b),\nabla^B(b)) +
(\rho \tau + \lambda)b^2]g_F }$.
\end{enumerate}
\end{theorem}

\begin{proof} Assume that $(M,g,\varphi,\rho,\lambda)$ is a gradient $\rho$-Einstein soliton
which is a also warped product. If $X,Y$ are
vector fields on $B$ and $V,W$ are vector fields on $F,$ then: \\
$\bullet$ By evaluating the equation (\ref{meqn:gradient-Yamabe}) at $(X,Y)$, we obtain:
$${\rm Ric}_B(X,Y) - \frac{s}{b} {\rm H}^b_B(X,Y)+
{\rm H}^{\varphi}_B(X,Y) = (\rho \tau + \lambda)g_B(X,Y).$$
The scalar curvature of warped product manifolds implies that the scalar curvature
of $(F,g_F),$ denoted by $\tau_F$ is constant. \\
$\bullet$ By evaluating the equation (\ref{meqn:gradient-Yamabe}) at $(X,V)$, we also obtain: \\
${\rm H}^{\varphi}(X,V)=0.$ So, $g_F({\rm d} \sigma(\nabla \varphi),V) = 0 $
for any vector field $V.$ Thus $\varphi$ depends on $B,$ i.e,
$\varphi \in \mathcal C^\infty(B).$ \\
$\bullet$ By evaluating the equation (\ref{meqn:gradient-Yamabe}) at $(V,W)$, we finally
obtain (4).
\end{proof}

\begin{remark} If $s={\rm dim}(F) \geq 3,$ then (4) of Theorem \ref{main-1}
can be stated as $(F,g_F)$ is Einstein with $\lambda_F,$ i.e, ${\rm Ric}_F=
\lambda_F g_F$ where $$\lambda_F=b^\sharp - bg_B(\nabla^B(b),\nabla^B(b)) +
(\rho \tau + \lambda)b^2.$$
\end{remark}

\section{Applications}

\subsection{Gradient $\rho$-Einstein Soliton on Generalized Robertson-Walker Space-times}

We first define generalized Robertson-Walker space-times. Assume that
$(F,g_F)$ is an $s-$dimensional Riemannian manifold and $b:I\rightarrow (0,\infty )$
is a smooth function. Then the $(s+1)-$dimensional product manifold $I\times _{b}F$
equipped with the metric tensor
\begin{equation*}
g=-\mathrm{d}t^{2}\oplus b^{2}g_F
\end{equation*}%
is called a generalized Robertson-Walker space-time and is denoted by
$M=I \times _{b}F$ where $I$ is an open and connected interval in $\mathbb{R}$
and $\mathrm{d}t^{2}$ is the usual Euclidean metric tensor on $I$. This structure
was introduced to extend Robertson-Walker space-times \cite%
{Sanchez98, Sanchez99} and have been studied by many authors, such
as \cite{ManticaDe1,ManticaDe2,Chen}. From now on, we will denote $\frac{\partial }{%
\partial t}\in \mathfrak{X}(I)$ by $\partial _{t}$ to state our results in
compact forms.

We will apply our main result Theorem \ref{main-1}. Assume that
$\varphi \in \mathcal C^\infty(I)$ is a potential function for a
generalized Robertson-Walker space-time of the form $M=I \times _{b}F.$

Hence, we can state that:

\begin{theorem} Suppose that $M=I \times _{b}F$ is a generalized Robertson-Walker
spacetime with the metric tensor $g=-{\rm d}t^2 \oplus b^2 g_F.$ Then
$(M,g,\varphi,\rho,\lambda)$ is a gradient $\rho$-Einstein soliton if and only if the
followings are satisfied
\begin{enumerate}
\item the potential function $\varphi$ depends only on the base manifold $I,$
\item the scalar curvature $\tau_F$ of the fiber manifold $(F,g_F)$ is constant,
\item $\displaystyle{\varphi^{\prime \prime} = -(\rho \tau + \lambda) +
s \frac{b^{\prime \prime}}{b^2}}$,
\item $\displaystyle{{\rm Ric}_F = [-bb^{\prime \prime} -(s-1)(b^\prime)^2 +
b(b^\prime)^2 + (\rho \tau + \lambda)b^2]g_F }$,
\end{enumerate}
where $\displaystyle{\tau = \frac{\tau_F}{b^2} + 2s \frac{b^{\prime \prime}}{b}
+s(s-1)\frac{(b^\prime)^2}{b^2} }$.
\end{theorem}

\subsection{Gradient $\rho$-Einstein Soliton on Standard Static Space-times}

We begin by defining standard static space-times. Let $(F,g_F)$ be an
$s-$dimensional Riemannian manifold and $f: F \rightarrow (0,\infty )$ be a smooth
function. Then the $(s+1)-$dimensional product manifold $_{f}I\times F$ furnished
with the metric tensor
\begin{equation*}
g=-f^{2}\mathrm{d}t^{2}\oplus g_F
\end{equation*}
is called a standard static space-time and is denoted by
$M=_{f} I\times F$ where $I$ is an open and connected subinterval in $\mathbb{R}$
and $dt^{2}$ is the usual Euclidean metric tensor on $I$.

Note that standard static space-times can be considered as a generalization
of the Einstein static universe\cite{AD1,AD,GES,Besse2008} and many spacetime
models that characterize the universe and the solutions of Einstein's field
equations are known to have this  structure.

Again we apply Theorem \ref{main-1}. Suppose that $\varphi \in
\mathcal C^\infty(F)$ is a potential function for a
standard static space-time of the form $M=_{f} I\times F.$

\begin{theorem} Let $M=_{f} I\times F$ be standard static spacetime with
the metric tensor $g=-f^2{\rm d}t^2 \oplus g_F.$ Then $(M,g,\varphi,\rho,\lambda)$
is a gradient $\rho$-Einstein soliton if and only if the followings are satisfied
\begin{enumerate}
\item the potential function $\varphi$ depends only on the fiber manifold $F,$
\item $\displaystyle{{\rm Ric}_F + {\rm Hess}^{\varphi}_F =
[(\rho \tau_F + \lambda)-2 \rho \frac{\Delta_F(f)}{f}]g_F + \frac{1}{f}{\rm Hess}^f_F}$,
\item $\displaystyle{-\nabla^F(f)+ \varphi(f) + 2 \rho \Delta_F(f) = [\rho \tau_F + \lambda]f }$.
\end{enumerate}
\end{theorem}

\begin{remark}
$\displaystyle{f {\rm Ric}_F - {\rm Hess}_F^f + f {\rm Hess}_F^{\varphi} =
[-\nabla^F(f) + \varphi(f)]g_F}$.
\end{remark}

\subsection{Gradient $\rho$-Einstein Soliton on 3-dimensional Walker Manifolds}

In general, a 3-dimensional manifold admitting a parallel degenerate
line field is said to be a Walker manifold, \cite{Walker1,Walker2}.
Suppose that $(M,g)$ is a 3-dimensional Walker Manifold then there
exist local coordinates $(t,x,y)$ such that the Lorentzian metric tensor
with respect to the local frame fields $\{ \partial_t, \partial_x, \partial_y \}$
takes the form given as:

\begin{equation}
\label{Walker metric}
g=2{\rm d}t{\rm d}y + {\rm d}x^2 + \phi(t,x,y) {\rm d}y^2,
\end{equation}
for some function $ \phi(t,x,y)$. The restricted case of Walker manifolds
where $\phi$ described as a function of only $x$ and $y$ is called as a
strictly Walker manifold and in particular strictly Walker manifolds are
geodesically complete. Also, it is known that a Walker manifold is Einstein
if and only if it is flat, \cite{Walker1}.

Now, we will investigate conditions on this particular class of manifolds
to have gradient Yamabe solitons, that is,
\begin{equation} \label{maineqn3WM} {\rm Ric} + {\rm Hess}(\varphi)_{ij}
= (\rho \tau+\lambda)g_{ij},
\end{equation}
where $\varphi$ is a potential function.

By using the metric $\eqref{Walker metric}$ and straightforward computations,
we have:
\begin{equation}
\label{system.3}
 \left\{ \begin{array}{ll}
{\rm Hess}(\varphi)_{tt} = \varphi_{tt}, \\
{\rm Hess}(\varphi)_{tx} = \varphi_{tx}, \\
{\rm Hess}(\varphi)_{ty} = \varphi_{ty} - \frac{1}{2} \phi_t \varphi_t, \\
{\rm Hess}(\varphi)_{xx} = \varphi_{xx}, \\
{\rm Hess}(\varphi)_{xy} = \varphi_{xy} - \frac{1}{2} \phi_x \varphi_t, \\
{\rm Hess}(\varphi)_{yy} = \varphi_{yy} - \frac{1}{2} (\phi \phi_t+\phi_y) \varphi_t +\frac{1}{2} \phi_x \varphi_x + \frac{1}{2} \phi_t
\varphi_y.
\end{array} \right.
\end{equation}

Moreover,

\begin{equation} \label{eqn3WM}
{\rm Ric}=\left[
              \begin{array}{ccc}
                0 & 0 & \frac{1}{2} \phi_{tt} \\
                0 & 0 & \frac{1}{2} \phi_{tx} \\
                \frac{1}{2} \phi_{tt} & \frac{1}{2} \phi_{tx} & \frac{1}{2} (\phi \phi_{tt} - \phi_{xx}) \\
              \end{array}
            \right]
\end{equation}

By combining these, we get:

By applying equations \eqref{maineqn3WM}, $\eqref{system.3}$ and $\eqref{eqn3WM}$
we obtain the following system of PDEs:

\begin{equation}
\label{system.4}
 \left\{ \begin{array}{ll}
\varphi_{tt}=0, \\
\varphi_{tx}=0,\\
\frac{1}{2} \phi_{tt}+\varphi_{ty}-\frac{1}{2}\phi_t \varphi_t=(\rho \tau+\lambda),\\
\varphi_{xx}=(\rho \tau+\lambda),\\
\frac{1}{2} \phi_{tx} + \varphi_{xy}-\frac{1}{2}\phi_x \varphi_t=0,\\
\frac{1}{2}(\phi \phi_{tt} - \phi_{xx})+ \varphi_{yy} - \frac{1}{2}(\phi \phi_t +\phi_y)\varphi_t +\frac{1}{2} \phi_x \varphi_x+ \frac{1}{2}
\phi_t \varphi_y=
(\rho \tau+\lambda) \phi
\end{array} \right.
\end{equation}

The first two equations of $\eqref{system.4}$ imply that
\begin{equation}
\label{wm_m1}
\varphi(t,x,y)=t B(y) + C(x,y) \quad \text{for some functions} \quad B,C.
\end{equation}

Combining $\eqref{system.4}$ and $\eqref{wm_m1},$ we obtain:

\begin{equation}
\label{system.4a}
 \left\{ \begin{array}{ll}
\frac{1}{2} \phi_{tt}+B^\prime(y)-\frac{1}{2} \phi_t B(y)=(\rho \tau+\lambda), \\
C_{xx}=(\rho \tau+\lambda),\\
\frac{1}{2} \phi_{tx} + C_{xy}(x,y)-\frac{1}{2} \phi_x B(y)=0,\\
\frac{1}{2}(\phi \phi_{tt} - \phi_{xx}) + t B^{\prime \prime}+C_{yy}(x,y)-
\frac{1}{2}(\phi \phi_t+\phi_y) B(y)\\+\frac{1}{2}(t B^\prime(y)+C_{y}(x,y))\phi_t+\frac{1}{2}C_x(x,y)\phi_x =(\rho \tau+\lambda) \phi
\end{array} \right.
\end{equation}

Differentiating the first and third equations of $\eqref{system.4a},$ with
respect to $x$ and $t,$ respectively, we have:
%$$(\rho \tau+\lambda)_x=0.$$
%Thus $$(\rho \tau+\lambda)=D(t,y) \quad \text{for some function} \quad D.$$

%Then the second equation of $\eqref{system.4a}$ becomes
$$C_x(x,y)=xD(t,y)+H(y).$$
%Hence
\begin{equation}
\label{wm_m1a}
C(x,y)=\frac{x^2}{2} D(y) + xE(y)+F(y) \quad \text{for some functions} \quad D,E,F.
\end{equation}
Thus,
\begin{equation}
\label{wm_m2a}
f(t,x,y)=tB(y)+\frac{x^2}{2} D(y) + xE(y)+F(y).
\end{equation}
Now, putting $\eqref{wm_m1a}$ into the system $\eqref{system.4a}$, to get a concerete solution we may impose the condition:
{``$\phi$ depends only on $x$"}. Then by the last equation of the system $\eqref{system.4a}$, we get $\phi_{xx}B(y)=0$. Thus, the following two cases are obtained:

\textbf{Case I:} $\phi_{xx}=0$, then from $\eqref{system.4a}$, the
potential function and the metric function are  found by
\begin{equation}
f(t,x,y)=\gamma t+(\alpha x +\beta)x+F(y), \qquad \phi=ax+b,
\end{equation}
 for some scalars $a,b,  \alpha, \beta, \gamma \in  \mathbb{R}$.

\textbf{Case II:} On the other hand, if $B(y)=0$, then  the
potential function and the metric function are found by
\begin{equation}
f(x,y)= mx +n\frac{y^2}{2}+py+r, \qquad  \phi =\frac{k}{m^2}e^{mx}+lx+s,
\end{equation}
for some scalars $k,l,m,n,p,r,s \in  \mathbb{R}$.

\begin{theorem} \label{thm:Walker}
Let $(M,g)$ be a 3-dimensional Lorenztian Walker manifold
equipped with metric:
$$g=2{\rm d}t{\rm d}y + {\rm d}x^2 + \phi(x) {\rm d}y^2.$$
Then $(M,g)$ is a gradient  $\rho$-Einstein soliton if and only if   the
potential function of the soliton and the metric function are  given by one of the following 2-cases:
\begin{enumerate}
\item
\begin{equation*}
f(t,x,y)=\gamma t+(\alpha x +\beta)x+F(y), \qquad \phi=ax+b,
\end{equation*}
 for some scalars $a,b,  \alpha, \beta, \gamma \in  \mathbb{R}$.
\item \begin{equation*}
f(x,y)= mx +n\frac{y^2}{2}+py+r, \qquad  \phi =\frac{k}{m^2}e^{mx}+lx+s,
\end{equation*}
for some scalars $k,l,m,n,p,r,s \in  \mathbb{R}$.
\end{enumerate}
\end{theorem}

Moreover, a pseudo-Riemannian manifold is said to be conformally symmetric if its Weyl tensor is parallel, i.e., $\nabla W = 0$,
\cite{derdzinski2007, derdzinski2009}.
It is known that any conformally symmetric Riemannian manifold
is either locally symmetric (i.e., $\nabla R = 0$) or locally conformally flat (i.e., $W = 0$).
In the nontrivial case ($\nabla W =0$  and $\nabla R \neq 0$, $W\neq 0$ ), the manifold $(M, g) $ is said to be essentially conformally
symmetric.

\begin{theorem} \cite{calvino2014}
A three-dimensional pseudo-Riemannian manifold is essentially conformally
symmetric if and only if it is a strict Lorentzian Walker manifold, locally isometric to
$(\mathbb{R}^3, (t, x, y), g_a)$, where
\begin{equation*}
    g_a=2{\rm d}t{\rm d}y+{\rm d}x^2+(x^3+a(y)x){\rm d}y^2,
\end{equation*}
for an arbitrary smooth function $a(y)$.
\end{theorem}

Therefore, by applying the above theorem for gradient  $\rho$-Einstein solitons, we terminate our study by the following result:

\begin{theorem}
There exists no 3-dimensional esentially conformally symmetric gradient  $\rho$-Einstein soliton.
\end{theorem}

\begin{proof}
By using system (\ref{system.4}), we have:
\begin{equation}
\label{system.8}
 \left\{ \begin{array}{ll}
\phi_{tt}=0, \\
\phi_{tx}=0, \\
\phi_{ty}=\lambda =\phi_{xx}, \\
\phi_{xy}-\frac{1}{2}[3x^2+a(y)] \phi_t = 0, \\
\phi_{yy}-\frac{1}{2}a^\prime(y) x \phi_t+\frac{1}{2}[3x^2+a(y)]\phi_x-\frac{x}{3}=\lambda[x^3+a(y)x]
\end{array} \right.
\end{equation}

First, $\phi_{tx}=0$ implies that $\phi_t=A(t,y).$ Also, $\phi_{tt}=A_t(t,y)=0.$ So, $A(t,y)=B(y).$
Then $\phi_t=B(y).$ Thus, $\phi=t B(y) + C(x,y)$ and $\lambda=B^\prime(y).$

Also, $\phi_{xx}=\lambda=B^\prime(y).$ So, $C_{xx}(x,y)=B^\prime(y).$ Then $C_x(x,y)=x B^\prime(y)+D(y).$
Hence, $C(x,y)=\frac{x^2}{2} B^\prime(y)+x D(y)+E(y)$. Thus,  we get \\
$\phi=tB(y)+\frac{x^2}{2}B^\prime(y)+xD(y)+E(y)$.

$\bullet$ $\phi_{xy}=\frac{1}{2}[3x^2+a(y)] \phi_t$ implies that
\begin{equation} \label{thm8_8}
xB^{\prime \prime}(y)+D^\prime(y)=\frac{1}{2}[3x^2+a(y)]B(y).
\end{equation}

$\bullet$ $\phi_{yy}-\frac{1}{2}a^\prime(y) x \phi_t+\frac{1}{2}
[3x^2+a(y)]\phi_x-\frac{x}{3}=\lambda[x^3+a(y)x]$ implies that
\begin{equation}
\begin{split} t B^{\prime \prime}(y) + \frac{x^2}{2}B^{\prime \prime \prime}(y)
+ x D^{\prime \prime}(y) + E^{\prime \prime}(y) \\
& = \frac{x}{2}a^\prime(y)B(y) \\
& - \frac{1}{2}[3x^2+a(y)][xB^\prime(y)+D(y)] \\
& + \frac{x}{3} + [x^3+ x a(y)]B^\prime(y).
\end{split}
\end{equation}
Differentiating the last two equations of System (\ref{system.8}) with
respect to $x$ and then combining the resulting equations, we obtain:
\begin{equation} \label{thm8_x}
\frac{3x^2}{2}B^\prime(y)+3xD(y)-\frac{1}{3}=\frac{a(y)}{2}B^\prime(y).
\end{equation}
By differentiating equation \eqref{thm8_x} with respect to $x,$ we have
$D^\prime(y)=-xB^{\prime \prime}(y).$
The last equation and the equation (\ref{thm8_8}) imply that \\
\begin{equation} \label{thm8_last}
[3x^2+a(y)]B(y)=0.
\end{equation}
Thus we have two cases:
\begin{itemize}
\item  $3x^2+a(y)=0$ cannot happen since $a$ is a  function of $y$.
\item If $B(y)=0,$ then in this case, we get $D(y)=0$ and thus $\lambda$ vanishes, which is a contradiction.
\end{itemize}
Therefore, the proof is completed.
\end{proof}

\end{document}